\documentclass[a4paper]{amsart}
\usepackage{amssymb}
\usepackage{amsmath}
\newcommand{\R}{\mathbb{R}}
\newcommand{\C}{\mathbb{C}}

 \DeclareMathOperator{\Ort}{O}

\newcommand{\D}{\mathrm{d}} 
 
\DeclareMathOperator{\Ree}{Re}
\newtheorem{theorem}{Theorem}[section]
\newtheorem{corollary}[theorem]{Corollary}
\newtheorem{lemma}[theorem]{Lemma}
\newtheorem{proposition}[theorem]{Proposition}
\theoremstyle{definition}

\newtheorem{remark}[theorem]{Remark}
\newtheorem{definition}[theorem]{Definition}
\theoremstyle{plain}

\numberwithin{equation}{section}

\usepackage{ifpdf}
\ifpdf
\usepackage[pdftex]{graphicx}
\DeclareGraphicsExtensions{.pdf}
\usepackage[pdftex,plainpages=false,pdfpagelabels,linktocpage,pdfstartview=FitH,colorlinks=true,linkcolor=blue,citecolor=magenta]{hyperref}
\else
\usepackage[dvips]{graphicx}
\usepackage[hypertex,linkcolor=cyan]{hyperref}
\fi

\title{Polynomial conserved quantities for constrained Willmore surfaces}
\author{A.C. Quintino and S.D. Santos}
\address{CENTRO DE MATEM\'{A}TICA, APLICA\c{C}\~{O}ES FUNDAMENTAIS E INVESTIGA\c{C}\~{A}O OPERACIONAL, FACULDADE DE CI\^{E}NCIAS DA UNIVERSIDADE DE LISBOA\\1749-016 LISBOA\\PORTUGAL}
\email{amquintino@ciencias.ulisboa.pt}

\address{CENTRO DE MATEM\'{A}TICA, APLICA\c{C}\~{O}ES FUNDAMENTAIS E INVESTIGA\c{C}\~{A}O OPERACIONAL, FACULDADE DE CI\^{E}NCIAS DA UNIVERSIDADE DE LISBOA\\1749-016 LISBOA\\PORTUGAL}
\email{sdsantos@ciencias.ulisboa.pt}

\begin{document}

\begin{abstract}
We define a hierarchy of special classes of constrained Willmore surfaces by means of the existence of a polynomial conserved quantity of some type, filtered by an integer. Type 1 with parallel top term characterises parallel mean curvature surfaces and, in codimension 1, type 1 characterises constant mean curvature surfaces. We show that this hierarchy is preserved under both spectral deformation and B\"{a}cklund transformation, for special choices of parameters, defining, in particular, transformations of constant mean curvature surfaces into new ones, with preservation of the mean curvature, in the latter case.
\end{abstract}

\maketitle

\section{Introduction}
Willmore surfaces are the critical points of the Willmore functional. A larger class arises when one imposes the weaker requirement that a surface is a critical point of the Willmore functional only with respect to infinitesimally conformal variations: these are the constrained Willmore surfaces. Constrained Willmore surfaces in the conformal sphere are characterised \cite{blaschke, BPP, BC, SD, ejiri,rigoli} by the (possibly perturbed) harmonicity of the central sphere congruence (following the terminology introduced in \cite{BQ}). The theory of harmonic maps then applies and, in particular, the machinery of integrable systems becomes available. First of all, a zero-curvature representation is established: a constrained Willmore surface comes \cite{BC} equipped with an associated family $\D^{\lambda}$ of flat connections, depending on a spectral parameter $\lambda\in\C\backslash\{0\}$. This structure gives rise to two kinds of symmetries: a spectral deformation \cite{BC,SD,BQ}, by exploiting a scaling freedom in the spectral parameter, and B\"{a}cklund transformations \cite{BQ}, which arise by applying chosen gauge transformations to the family of flat connections.

Alike what happens in the case of constrained Willmore surfaces, the isothermic surface condition amounts \cite{BurDonPedPin09} just as well to the flatness of a certain family $\nabla^{t}$ of connections, indexed in $\R$. In \cite{BS}, the classical notion of special isothermic surface,
introduced by Darboux in connection with deformations of quadrics, is given a simple explanation in terms of the integrable
systems approach to isothermic surfaces. They are realised as a
particular case of a hierarchy of classes of isothermic surfaces
filtered by an integer $d$. Here is the basic idea: The theory of ordinary differential equations
ensures that we can find $\nabla^{t}$-parallel sections depending
smoothly on the spectral parameter $t$. The existence of such
sections with polynomial dependence of degree $d$ on $t$ is of
particular geometric significance, as first observed by F. E.
Burstall and D. Calderbank (see the forthcoming paper
"Conformal submanifold geometry IV-V"), and gave rise to the notion
of polynomial conserved quantity of type $d$, developed in \cite{BS}, in the
isothermic context, where the notion of special isothermic surface of type $d$ is introduced, having the classical notion as a particular case ($d=2$).

We are in this way led to the central idea of
this article, that of \emph{special constrained Willmore surface of type $d$}, a constrained Willmore surface admitting some \textit{polynomial conserved quantity of type $d$} in the
constrained Willmore context, that is, a certain family $p(\lambda)$ of $\D^{\lambda}$-parallel sections with Laurent polynomial dependence on $\lambda$, with degree smaller or equal to $d$.

At the intersection of the class of constrained Willmore surfaces with the class of
isothermic surfaces lies, in particular, that of non-zero parallel mean
curvature vector surfaces in space-forms, and that of constant mean
curvature surfaces in $3$-dimensional space-forms. In the isothermic context, type $1$ characterises \cite{BurCal,BS} the $H$-generalised surfaces in space-forms. We prove that, in the constrained Willmore context, type $1$ with parallel top term characterises parallel mean curvature vector surfaces in space-forms. It follows, in particular, that, in codimension $1$, type $1$ characterises constant mean curvature surfaces, in both contexts.

We prove that the class of constrained Willmore surfaces of any given type is preserved by both spectral deformation and B\"{a}cklund transformation, for special
choices of parameters. Both constrained Willmore spectral deformation and B\"{a}cklund transformation prove to preserve also the parallelism of the top term of a polynomial conserved quantity. For the particular case of type $1$, this defines transformations of parallel mean curvature surfaces into new ones, and, in the particular case of codimension $1$, transformations of constant mean curvature surfaces into new ones, with preservation of both the space-form and the mean curvature, under B\"{a}cklund transformation.

Our theory is local and, throughout the text, with no need for further reference, restriction to a suitable non-empty open set shall be underlying.

The results of this paper are based, in part, on those in the first
author's PhD thesis \cite{aurea1,aurea} and some of them were announced in \cite{aurea2}.\newline

\textbf{Acknowledgements.} The authors would like to thank Rui Pacheco for helpful conversations. Very special thanks are due to Fran Burstall and David Calderbank, who first observed the particular geometric significance of polynomial dependence on a parameter for some families of parallel sections and who have had a decisive influence on the origin of this paper.

The first-named author was supported in the research work carried out in \cite{aurea1,aurea} by Funda\c{c}\~{a}o para a Ci\^{e}ncia e a Tecnologia, Portugal, with a PhD scholarship, and by Funda\c{c}\~{a}o da Faculdade de Ci\^{e}ncias da Universidade de Lisboa, with a postdoctoral scholarship.

\section{Constrained Willmore surfaces in the conformal $n$-sphere}
Consider $\underline{\C}^{n+2}=\Sigma\times(\R^{n+1,1})^{\C}$
provided with the complex bilinear extension of the metric on
$\underline{\R}^{n+1,1}$. In what follows, we may abuse notation and make no explicit
distinction between a bundle and its complexification. Throughout this text, we consider the identification
$$\wedge^{2}\R^{n+1,1}\cong o(\R^{n+1,1})$$ of the exterior power
$\wedge^{2}\R^{n+1,1}$ with the orthogonal algebra $o(\R^{n+1,1})$
via $$u\wedge v(w):=(u,w)v-(v,w)u$$ for $u,v,w\in\mathbb R^{n+1,1}$.

\subsection{Conformal submanifold geometry}

Our study is one of surfaces in $n$-dimensional space-forms, with
$n\geq 3$, from a conformally-invariant viewpoint. For this, we
find a convenient setting in Darboux's light-cone model of the
conformal $n$-sphere \cite{darbouxsphere}. We follow
the modern account presented in \cite{Bur06}. So contemplate the
light-cone $\mathcal{L}$ in the Lorentzian vector space $\R^{n+1,1}$
and its projectivisation $\mathbb{P}(\mathcal{L})$, provided with
the conformal structure defined by a metric $g_{\sigma}$ arising
from a never-zero section $\sigma$ of the tautological bundle
$\pi:\mathcal{L}\rightarrow\mathbb{P}(\mathcal{L})$ via
$g_{\sigma}(X,Y)=(\D\sigma (X), \D\sigma(Y))$. For
$v_{\infty}\in\R^{n+1,1}_{\times}$, set
$$S_{v_{\infty}}:=\{v\in\mathcal{L}:(v,v_{\infty})=-1\},$$ an
$n$-dimensional submanifold of $\mathbb R^{n+1,1}$ which inherits from
$\R^{n+1,1}$ a positive definite metric of (constant) sectional
curvature $-(v_{\infty},v_{\infty})$. By construction, the bundle
projection $\pi$ restricts to give a conformal diffeomorphism
$\pi_{\vert{S_{v_{\infty}}}}:S_{v_{\infty}}\rightarrow
\mathbb{P}(\mathcal{L})\backslash\mathbb{P}(\mathcal{L}\cap\langle
v_{\infty}\rangle^{\perp})$. In particular, choosing $v_{\infty}$ to be
time-like identifies $\mathbb{P}(\mathcal{L})$ with the conformal
$n$-sphere, $$S^{n}\cong\mathbb{P}(\mathcal{L}).$$

For us, a mapping $\Lambda:\Sigma\rightarrow
\mathbb{P}(\mathcal{L})$, of a surface $\Sigma$, is the
same as a null line subbundle of the trivial bundle
$\underline{\mathbb{R}}^{n+1,1}=\Sigma\times\mathbb{R}^{n+1,1}$.
Let then $\Lambda:\Sigma\rightarrow
\mathbb{P}(\mathcal{L})$ be an immersion of an oriented
surface $\Sigma$, which we provide with the conformal structure
$\mathcal{C}_{\Lambda}$ induced by $\Lambda$ and with the canonical
complex structure. Set
$$\Lambda^{1,0}:=\Lambda\oplus
\D\sigma(T^{1,0}M),\,\,\,\,\Lambda^{0,1}:=\Lambda\oplus
\D\sigma(T^{0,1}M),$$ defined independently of the choice of
$\sigma\in\Gamma(\Lambda)$ never-zero, and then
$\Lambda^{(1)}:=\Lambda^{1,0}+\Lambda^{0,1}$. Let
$S:\Sigma\rightarrow \mathcal{G}:=\mathrm{Gr}_{(3,1)}(\mathbb
R^{n+1,1})$ be the central sphere congruence of $\Lambda$,
$$S=\Lambda^{(1)}\oplus\langle\triangle\sigma\rangle,$$for $\sigma$
a lift of $\Lambda$ and $\triangle\sigma$ the Laplacian of $\sigma$,
with respect to the metric $g_{\sigma}$. We have a decomposition
$\underline{\mathbb R}^{n+1,1}=S\oplus S^{\perp}$ and then a
decomposition of the trivial flat connection $\D$ on
$\underline{\mathbb R}^{n+1,1}$ as $$\D=\mathcal{D}\oplus\mathcal{N},$$ for
$\mathcal{D}$ the connection given by the sum of the connections
induced on $S$ and $S^{\perp}$ by $\D$. At times, it will be
convenient to make en explicit reference to the surface $\Lambda$,
writing $S_{\Lambda}$, $\mathcal{D}_{\Lambda}$ (or, equivalently, $\mathcal{D}_{S}$) and $\mathcal{N}_{\Lambda}$ (or, equivalently, $\mathcal{N}_{S}$) for $S$, $\mathcal{D}$ and $\mathcal{N}$, respectively. For later reference, we
define analogously $\mathcal{D}^{\hat{\D}}$ and $\mathcal{N}^{\hat{\D}}$, for a general connection
$\hat{\D}$ on $\underline{\mathbb{C}}^{n+2}$.

\subsection{Constrained Willmore surfaces and flat connections}
Willmore surfaces are characterised \cite{blaschke,ejiri,rigoli} by the harmonicity of the central sphere congruence. More generally:
\begin{theorem}\cite{BPP, BC, SD} \label{CWtemp}
$\Lambda$ is a constrained Willmore surface if and only if there exists a real form
$q\in\Omega ^{1}(\Lambda\wedge\Lambda ^{(1)})$ with
\begin{equation}\label{eq:curlyDextderivofq}
\D^{\mathcal{D}}q=0
\end{equation}
such that
\begin{equation}\label{eq:mainCWeq}
\D^{\mathcal{D}}*\mathcal{N}=2\,[q\wedge *\mathcal{N}],
\end{equation}
where $[\,,\,]$ denotes the $2$-form defined from the Lie Bracket $[\,\,,\,]$ in $o(\R^{n+1,1})$.
\end{theorem}
The introduction of a constraint in the variational problem equips surfaces $\Lambda$ with \emph{Lagrange multipliers} $q$, defining pairs $(\Lambda,q)$. Willmore surfaces are
the constrained Willmore surfaces admitting the zero multiplier.

For maps into a Grassmannian, harmonicity amounts \cite{uhlenbeck} to the flatness of a  family of connections. Ultimately, and crucially, a zero-curvature characterisation of constrained Willmore surfaces follows:

\begin{theorem}\cite{BC}\label{constrharmonicity}
Given a real form $q\in\Omega^{1}(\Lambda\wedge\Lambda^{(1)})$,
$(\Lambda,q)$ is a constrained Willmore surface if and only if the connection
\begin{equation}\label{eq:dlamdaqV}
\D^{\lambda}_{q}:=\mathcal{D}+\lambda\mathcal{N}^{1,0}+\lambda^{-1}\mathcal{N}^{0,1}+(\lambda^{2}-1)q^{1,0}+(\lambda^{-2}-1)q^{0,1}
\end{equation}
is flat, for all $\lambda\in\C\backslash\{0\}$.
\end{theorem}

At times, it will be convenient to make an explicit reference to the
surface $\Lambda$, writing $\D^{\lambda,q}_{\Lambda}$ (or, equivalently, $\D^{\lambda,q}_{S}$)  for $\D^{\lambda}_{q}$.

The isothermic surface condition amounts just as well to the flatness of a certain family
of connections:

\begin{proposition}\cite{BurDonPedPin09,Her03}\label{iso}
$\Lambda$ is isothermic if and only if there exists a non-zero
$1$-form $\eta\in\Omega^{1}(\Lambda\wedge\Lambda^{(1)})$ such that $\D\eta=0$
or, equivalently, such that the connection
$$\nabla^{t}:=\D+t\eta$$ is flat, for all $t\in\R$.
\end{proposition}

In the conditions of Proposition \ref{iso}, we will say that $(\Lambda,\eta)$ is an isothermic surface. In the case
$\Lambda$ is not contained in any $2$-sphere, the form $\eta$ is
unique up to a (non-zero) constant real scale, cf. \cite{Bur06}. As
for a given constrained Willmore surface, the multiplier $q$ is, in general, unique, the exception being when $\Lambda$ is, in addition, isothermic:

\begin{proposition}\cite{BQ}\label{uniqq}
A constrained Willmore surface admits a unique multiplier if and only if it is not an
isothermic surface. Furthermore: if $(\Lambda,\eta)$ is an
isothermic constrained Willmore surface admitting a multiplier $q$, then the set of
multipliers to $\Lambda$ is the affine space $q+\langle
*\eta\rangle_{\R}$.
\end{proposition}

\subsection{Transformations of constrained Willmore surfaces}

Constrained Willmore surfaces come equipped with a family of flat
connections. Transformations of this family have been exploited
\cite{BC,SD,BQ} in order to produce new
constrained Willmore surfaces, as we recall next.

Suppose $(\Lambda,q)$ be a constrained Willmore surface.
\subsubsection{Spectral deformation}
For each $\lambda$ in $S^{1}$, $\D^{\lambda}_{q}$ is a real flat
metric connection, which establishes the existence of an isometry
$\phi^{\lambda}_{q}:(\underline{\mathbb R}^{n+1,1},
\D^{\lambda}_{q})\rightarrow(\underline{\mathbb R}^{n+1,1},\D)$ of
bundles, defined on a simply connected component of $\Sigma$,
preserving connections, unique up to a M\"{o}bius transformation. We
define a spectral deformation of $\Lambda$ into new constrained Willmore surfaces by
setting, for each $\lambda$ in $S^{1}$,
$$\Lambda^{\lambda}_{q}:=\phi^{\lambda}_{q}\Lambda,$$ cf.
\cite{BC,SD,BQ}.

\subsubsection{B\"{a}cklund transformation}

In \cite{BQ}, a version of the Terng-Uhlenbeck dressing action \cite{TU} is used to construct new constrained Willmore surfaces from a given one, as follows. Let $\rho\in\Gamma(\underline{\mathbb R}^{n+1,1})$ be reflection
across $S$, $\rho=\pi_{S}-\pi_{S^{\perp}}$, for $\pi_{S}$ and
$\pi_{S^{\perp}}$ the orthogonal projections of $\underline{\mathbb R}^{n+1,1}$ onto $S$ and $S^{\perp}$,
respectively. Let $\alpha\in\C\backslash S^{1}$ be non-zero and $L\subset\underline{\C}^{n+2}$ be a null line bundle such that $L$ and $\rho L$ are never orthogonal. Define transformations $p^{(-)}_{\alpha, L}(\lambda)\in\Gamma(O(\underline\C^{n+2}))$, for $\lambda\in\C\backslash\{\pm\alpha\}$, by\footnote{In \cite{BQ}, an extra factor is introduced in the eigenvalues of $p_{\alpha,L}(\lambda)$, resulting in the normalization of the family $\lambda\mapsto p_{\alpha,L}(\lambda)$, $p_{\alpha,L}(1)=I$.}
\begin{equation*}
p^{(-)}_{\alpha,L}(\lambda)=
\begin{cases}
(-)\frac{\lambda-\alpha}{\lambda+\alpha}&\text{on $L$;}\\
1&\text{on $(L\oplus \rho L)^{\perp}$;}\\
(-)\frac{\lambda+\alpha}{\lambda-\alpha}&\text{on $\rho L$;}
\end{cases}
\end{equation*}
respectively. We define in this way two maps of $\C\backslash\{\pm\alpha\}$ into $\Gamma(O(\underline\C^{n+2}))$ that extend holomorphically to the Riemann sphere except $\pm\alpha$, by setting $p_{\alpha,L}(\infty):=I$ and
\begin{equation*}
p^{-}_{\alpha,L}(\infty)=
\begin{cases}
-1&\text{on $L$;}\\
1&\text{on $(L\oplus \rho L)^{\perp}$;}\\
-1&\text{on $\rho L$.}
\end{cases}
\end{equation*}
It will be useful to note that $p_{\alpha,L}(\lambda)=p^{-}_{\alpha^{-1},L}(\lambda^{-1})$, for all $\lambda\in\C\backslash\{0,\pm\alpha\}$. For further reference, observe also that
\begin{equation}\label{eq:rhopq}
p^{(-)}_{\alpha,L}(-\lambda)=\rho p^{(-)}_{\alpha,L}(\lambda)\rho^{-1},
\end{equation}
respectively, for all $\lambda\in\mathbb{P}^{1}\backslash\{\pm\alpha\}$. Set $\hat{\alpha}:=\overline{\alpha}\,^{-1}$, $\tilde{L}=p_{\hat{\alpha},\overline{L}}(\alpha)L$ and $r=p^{-}_{\alpha,\tilde{L}}p_{\hat{\alpha},\overline{L}}$. Consider the transform $$\hat{S}:=r(1)^{-1}S$$ of $S$ and the transforms $$\hat{\Lambda}^{1,0}:=r(1)^{-1}r(0)\Lambda^{1,0},
\,\,\,\,\hat{\Lambda}^{0,1}:=r(1)^{-1}r(\infty)\Lambda^{0,1},$$of $\Lambda^{1,0}$ and $\Lambda^{0,1}$, respectively, by $r$. Set $\hat{\Lambda}:=\hat{\Lambda}^{1,0}\cap\hat{\Lambda}^{0,1}$.

\begin{theorem}\cite{BQ}\label{BT}
$\hat{\Lambda}$ is a surface with central sphere congruence $\hat{S}$. Define, furthermore, $\tilde{q}\in\Omega^{1}(\wedge^{2}S\oplus\wedge^{2}S^{\perp})$ by
$$\tilde{q}^{1,0}:=r(\infty)q^{1,0}r(\infty)^{-1},\,\,\,\,\tilde{q}^{0,1}:=r(0)q^{0,1}r(0)^{-1}$$ and set $$\hat{q}:=r(1)^{-1}\tilde{q}r(1).$$ Then $(\hat{\Lambda},\hat{q})$ is a constrained Willmore surface.
\end{theorem}

$(\hat{\Lambda},\hat{q})$ is said to be the \emph{B\"{a}cklund transform of $(\Lambda,q)$ of parameters $\alpha, L$}.\footnote{We are, in fact, considering the reparametrization of the B\"{a}cklund transformation presented in \cite{BQ} that results of interchanging parameters $\alpha, L$ with parameters $\hat{\alpha},\overline{L}$.}

\section{Special constrained Willmore surfaces of type $d$} \label{specialsection}

In \cite{BS}, the concept of special isothermic surface of type $d$ is introduced, as an isothermic surface admitting a polynomial conserved quantity of degree $d$, having the classical notion of special isothermic surface as the particular case of $d=2$. The focus of
this article will be the study of \emph{special constrained Willmore surfaces}
and, in particular, that of a polynomial conserved quantity in the
constrained Willmore context.

Let $\Lambda$ be a surface in $S^{n}\cong\mathbb{P}(\mathcal{L})$.

\subsection{Polynomial conserved quantities for constrained Willmore surfaces}
Suppose $(\Lambda,q)$ be a constrained Willmore surface, with its
associated family of flat connections $\D^{\lambda}_{q}$, indexed by
$\C\backslash\{0\}$.

\begin{definition}\label{definition} Let $p(\lambda)=\sum_{k=-d}^d p_{k}\lambda^{k}$ be a Laurent
polynomial with coefficients in
$\Gamma(\underline{\mathbb{C}}^{n+2})$, with $d\in\mathbb{N}_{0}$,
such that:
\begin{subequations}\label{polynomial}
\begin{gather}
\label{eq:1}
p_{-k}=\overline{p_{k}}\mbox{, for all } k;\\
\label{eq:2}
p_{d}\in\Gamma(S^{\perp});\\
\label{eq:3} p_{k}\in\Gamma(S^{\perp})\mbox{ if }k\mbox{ and }d\mbox{ have the
same parity;
otherwise }p_{k}\in\Gamma(S);\\
\label{eq:4} p(1)\neq0.
\end{gather}
\end{subequations}
We say that $p(\lambda)$ is a \emph{polynomial conserved quantity of
$(\Lambda,q)$ of type $d$} if
\begin{equation}\label{eq:pcq}
\D^{\lambda}_{q}p(\lambda)=0,
\end{equation}
for all $\lambda\in\mathbb{C}\backslash\{0\}$. We say that $(\Lambda,q)$ is a \emph{special constrained Willmore surface of type $d$}
if it admits a polynomial conserved quantity of type $d$.
\end{definition}

Constrained Willmore surfaces in space-forms constitute a conformally-invariant class of
surfaces and so does the class of special constrained Willmore surfaces of type $d$,
given $d\in\mathbb{N}_{0}$:
\begin{proposition}
Let $d$ be in $\mathbb{N}_{0}$ and $T$ be in
$\Ort(\mathbb{R}^{n+1,1})$. Suppose $(\Lambda,q)$ is a constrained Willmore surface of
type $d$. Then so is $(T\Lambda,Ad_{T}(q))$.
\begin{proof}
The fact that $\mathcal{D}_{T\Lambda}=T\circ
\mathcal{D}_{\Lambda}\circ T^{-1}\mbox{ and
}\mathcal{N}_{T\Lambda}=T\circ \mathcal{N}_{\Lambda}\circ T^{-1}$
makes it clear that $(T\Lambda,Ad_{T}(q))$ is a constrained Willmore surface and,
furthermore, that, if $p(\lambda)=\sum_{k=-d}^d p_{k}\lambda^{k}$ is
a polynomial conserved quantity of $(\Lambda,q)$ of type $d$, then
$$s(\lambda):=T\circ p(\lambda)=\sum_{k=-d}^d T(p_{k})\lambda^{k}$$ is
a polynomial conserved quantity of $(T\Lambda,Ad_{T}(q))$ of type
$d$.
\end{proof}
\end{proposition}

\begin{proposition}\label{iff condition} A Laurent polynomial $p(\lambda)=\sum_{k=-d}^d p_{k}\lambda^{k}$ satisfying
the conditions \eqref{polynomial} is a polynomial conserved quantity
of $(\Lambda,q)$ if and only if
\begin{equation}\label{iff condition-equation}
\mathcal{D}
p_{k}+\mathcal{N}^{1,0}p_{k-1}+\mathcal{N}^{0,1}p_{k+1}+q^{1,0}p_{k-2}+q^{0,1}p_{k+2}-qp_{k}=0,\;\forall
k\in\{0,...,d+2\},
\end{equation}
with the convention $$p_{-d-4}=p_{-d-3}=p_{-d-2}=p_{-d-1}=p_{d+1}=p_{d+2}=p_{d+3}=p_{d+4}=0.$$

\begin{proof}
The result is a direct consequence of the fact that $\D
^{\lambda,q}p(\lambda)=0$ if and only if
\begin{multline*}
\sum_{k=-d}^d \mathcal{D}p_{k}\lambda^{k}+\sum_{k=-d+1}^{d+1} \mathcal{N}^{1,0}p_{k-1}\lambda^{k}+\sum_{k=-d-1}^{d-1} \mathcal{N}^{0,1}p_{k+1}\lambda^{k}+\\
\sum_{k=-d+2}^{d+2}q^{1,0}p_{k-2}\lambda^{k}+\sum_{k=-d-2}^{d-2} q^{0,1}p_{k+2}\lambda^{k}+\sum_{k=-d}^d
(-q^{1,0}-q^{0,1})p_{k}\lambda^{k}=0,
\end{multline*}
for all $\lambda\in\mathbb{C}\backslash\{0\}$.
\end{proof}
\end{proposition}

\begin{proposition}\label{consequences}If
$p(\lambda)=\sum_{k=-d}^d p_{k}\lambda^{k}$ is a polynomial
conserved quantity of $(\Lambda,q)$ of type $d\in\mathbb{N}_{0}$,
then
\begin{enumerate}
\item $p_{0}$ is real;
\item $p(1)$ is real and constant;
\item $\mathcal{D}^{0,1}p_{d}=0$;
\item $\mathcal{D}^{1,0}p_{d}+\mathcal{N}^{1,0}p_{d-1}=0$;
\item $\mathcal{N}^{1,0}p_{d}+q^{1,0}p_{d-1}=0$;
\item the Laurent polynomial $(p(\lambda),p(\lambda))$ has constant (complex) coefficients.
\end{enumerate}
\end{proposition}
\begin{proof}
The reality of $p_{0}$ and $p(1)$ is a consequence of
$\overline{p_{0}}=p_{0}$ together with
$p(1)=p_{0}+2\sum_{k=1}^d\Ree(p_{k})$. The constancy of $p(1)$ is
immediate from evaluating $\D^{\lambda}_{q}p(\lambda)$ at
$\lambda=1$. On the other hand, for $k=d$, equation \eqref{iff
condition-equation} establishes $\mathcal{D}^{0,1}p_{d}=0$ and
$\mathcal{D}^{1,0}p_{d}+\mathcal{N}^{1,0}p_{d-1}=0$ (note that
$qp_{d}=0$ and $qp_{d-2}=0$, as
$p_{d},p_{d-2}\in\Gamma(S^{\perp})$), whereas, taking $k=d+1$, we
get $\mathcal{N}^{1,0}p_{d}+q^{1,0}p_{d-1}=0$. The fact that
$\D^{\lambda}_{q}$ is a metric connection, together with equation
\eqref{eq:pcq}, establishes
$$\D(p(\lambda),p(\lambda))=2(\D^{\lambda}_{q}p(\lambda),p(\lambda))=0,$$
for all $\lambda\in\mathbb{C}\backslash\{0\}$, showing that the
polynomial $(p(\lambda),p(\lambda))$ has constant coefficients and
completing the proof.
\end{proof}

According to Proposition \ref{consequences}, the existence of a
polynomial conserved quantity $p(\lambda)$ establishes, in
particular, $p(1)$ as a non-zero real vector in $\R^{n+1,1}$,
defining therefore a space-form $S_{p(1)}$, which will be of
particular geometric relevance, as we shall see later.

\medskip

\begin{proposition} If $(\Lambda,q)$ admits a polynomial conserved quantity $p(\lambda)$
of type $d$, then it admits a polynomial conserved quantity
$s(\lambda)$ of type $d+1$ with $s(1)=p(1)$.
\begin{proof}
Take $$s(\lambda):=\frac{1}{2}(\lambda^{-1}+\lambda)p(\lambda)=\frac{1}{2}\sum_{k=-d-1}^{d+1}(p_{k-1}+p_{k+1})\lambda^{k},$$
for $p(\lambda)=\sum_{k=-d}^d p_{k}\lambda^{k}$ and under the
convention $p_{-d-2}=p_{-d-1}=p_{d+1}=p_{d+2}=0$.
\end{proof}
\end{proposition}

\subsection{Non-full constrained Willmore surfaces}

We say that a surface in $S^{n}$ is \emph{full} if it
does not lie in any proper sub-sphere of $S^{n}$. In the isothermic
context, type $0$ characterises \cite{BS} the surfaces which are not
full. This is also the case in the constrained Willmore context:

\begin{proposition}\label{type0}
$(\Lambda,q)$ is a special constrained Willmore surface of type $0$ if and only if
$\Lambda$ is not full.
\end{proposition}
\begin{proof}
Suppose $(\Lambda,q)$ admits a polynomial conserved quantity
$p(\lambda)=p_{0}$. Then $p_{0}$ is a non-zero, real constant section of
$S^{\perp}$ and $\Lambda$ lies therefore in the $(n-1)$-sphere
$\mathbb{P}(\mathcal{L}\cap\langle p_{0}\rangle^{\perp})$.
Conversely, if $\Lambda$ takes values in some sub-sphere, say
$\mathbb{P}(\mathcal{L}\cap\langle u\rangle^{\perp})$, where $u$ is
a positive definite vector of $\mathbb{R}^{n+1,1}$, we have
$u\in\Gamma(S^{\perp})$. Hence $p(\lambda):=u$ satisfies the
conditions \eqref{polynomial} and, for all
$\lambda\in\mathbb{C}\backslash\{0\}$,
$$\D^{\lambda}_{q}p(\lambda)=(\mathcal{D}+\lambda\mathcal{N}^{1,0}+\lambda^{-1}\mathcal{N}^{0,1})p_{0}=0,$$
as $p_{0}\in\Gamma(S^{\perp})$ and $\D p_{0}=0$ (or, equivalently,
$\mathcal{D}p_{0}=0=\mathcal{N}p_{0}$).
\end{proof}

\subsection{Surfaces with parallel mean curvature vector}

In the isothermic context, type $1$ characterises \cite{BurCal} (see also \cite{BS}), in general, the
$H$-generalised surfaces in some space-form. In the constrained Willmore context, type
$1$ with parallel top term characterises surfaces with parallel mean curvature vector, as we shall see in this section. In particular, for codimension $1$, type 1 characterises surfaces with constant mean curvature vector in both contexts.

Given $v_{\infty}\in\R^{n+1,1}$ such that $v_{\infty}\notin\Gamma(\Lambda^{\perp})$ and $\sigma\in\Gamma(\Lambda)$
never-zero, $\Lambda$ defines a local immersion
$$\sigma_{\infty}:=(\pi\vert_{S_{v_{\infty}}})^{-1}\circ
\Lambda=-\frac{1}{(\sigma,v_{\infty})}\,\sigma:\Sigma\rightarrow
S_{v_{\infty}},$$ of $\Sigma$ into the space-form $S_{v_{\infty}}$. Let $v_{\infty}^{T}$ and $v_{\infty}^{\perp}$ denote the orthogonal projections of $v_{\infty}$ onto $S$ and $S^{\perp}$, respectively. Consider the normal bundle $V_{v_{\infty}}^{\perp}=(\Lambda^{(1)}\oplus\langle v_{\infty}\rangle)^{\perp}$ of $\sigma_{\infty}$ and let $\mathbf{H}$ denote the mean curvature vector of $\sigma_{\infty}$. The map $\mathcal{Q}:V_{v_{\infty}}^{\perp}\longrightarrow S^{\perp}$ defined by $$\mathcal{Q}:\xi\mapsto(\mathbf{H},\xi)\sigma_{\infty}+\xi$$ is an isomorphism of bundles preserving connections.
Note that, as $$(v_{\infty}^{\perp},(\mathbf{H},\xi)\sigma_{\infty}+\xi)=-(\mathbf{H},\xi)=(-(\mathbf{H},\mathbf{H})\sigma_{\infty}-\mathbf{H},(\mathbf{H},\xi)\sigma_{\infty}+\xi),\,\forall\xi\in\Gamma(V_{v_{\infty}}^{\perp})$$ we have
\begin{equation}\label{QHvinftperp}
\mathcal{Q}\mathbf{H}=-v_{\infty}^{\perp}
\end{equation}
and then $\nabla^{V_{v_{\infty}}^{\perp}}\mathbf{H}=0$ (i.e., $\mathbf{H}$ parallel) if and only if $\mathcal{D}v_{\infty}^{\perp}=0$.

In this section, we first recall that parallel mean curvature vector surfaces are,
indeed, examples of constrained Willmore surfaces, proving it in our setting and
providing a Lagrange multiplier.

\begin{proposition}\label{multiplier}
Suppose $\Lambda$ has parallel mean curvature vector in some
space-form $S_{v_{\infty}}$. Then $(\Lambda,q_{\infty})$ is a constrained Willmore
surface, where
$$q_{\infty}:=\frac{1}{2}\,\sigma_{\infty}\wedge\mathcal{N}v_{\infty}^{\perp}=\frac{1}{2}\,\sigma_{\infty}\wedge \D
v_{\infty}^{\perp}.$$
\end{proposition}
\begin{proof}
First of all, note that the real $1$-form $q$ takes values in $\Lambda\wedge\Lambda^{(1)}$, as $\mathcal{N}v_{\infty}^{\perp}$ lives in $S$ and $$(\mathcal{N}v_{\infty}^{\perp},\sigma_{\infty})=-(v_{\infty}^{\perp},\mathcal{N}\sigma_{\infty})=-(v_{\infty}^{\perp},0)=0.$$

Taking into account that
\begin{equation*}
\begin{split}
2\D^{\mathcal{D}}q_{\infty}\Big(\frac{\partial}{\partial z},\frac{\partial}{\partial \bar{z}}\Big)&=(\sigma_{\infty})_{z}\wedge (v_{\infty}^{\perp})_{\bar{z}}-(\sigma_{\infty})_{\bar{z}}\wedge (v_{\infty}^{\perp})_{z}+\sigma_{\infty}\wedge \pi_{S}\big((v_{\infty}^{\perp})_{\bar{z}z}-(v_{\infty}^{\perp})_{z\bar{z}}\big)\\
&=(\sigma_{\infty})_{z}\wedge (v_{\infty}^{\perp})_{\bar{z}}-(\sigma_{\infty})_{\bar{z}}\wedge (v_{\infty}^{\perp})_{z},
\end{split}
\end{equation*}
for $\pi_{S}$ the orthogonal projection of $\underline{\mathbb R}^{n+1,1}$ onto $S$,
and noting that $(v_{\infty}^{\perp})_{z}\in\Gamma(\langle (\sigma_{\infty})_{\bar{z}}\rangle)$ and $(v_{\infty}^{\perp})_{\bar{z}}\in\Gamma(\langle (\sigma_{\infty})_{z}\rangle)$, by virtue of $(v_{\infty}^{\perp})_{z}$ being orthogonal to $v_{\infty}$ and $(\sigma_{\infty})_{\bar{z}}$, and $(v_{\infty}^{\perp})_{\bar{z}}$ being orthogonal to $v_{\infty}$ and $(\sigma_{\infty})_{z}$, we conclude that $\D^{\mathcal{D}}q_{\infty}=0$.

Now let us prove that $\D^{\mathcal{D}}\ast\mathcal{N}=2[q_{\infty}\wedge\ast\mathcal{N}]$. For that, first observe that $\D^{\mathcal{D}}\ast\mathcal{N}$ and $[q_{\infty}\wedge\ast\mathcal{N}]$ are determined by the respective restrictions to $S$, as both are $2$-forms taking values in $S\wedge S^{\perp}$. It is therefore enough to prove the equality in $S$.

For each $\xi\in\Gamma(\underline{\mathbb{R}}^{n+1,1})$,
$$\D^{\mathcal{D}}\ast\mathcal{N}\Big(\frac{\partial}{\partial z},\frac{\partial}{\partial \bar{z}}\Big)\xi
=2i\big(\mathcal{D}_{z}(\mathcal{N}_{\bar{z}}\xi)-\mathcal{N}_{\bar{z}}(\mathcal{D}_{z}\xi)+\mathcal{D}_{\bar{z}}(\mathcal{N}_{z}\xi)-\mathcal{N}_{z}(\mathcal{D}_{\bar{z}}\xi)\big).$$
In view of the flatness of $\D$, characterised by
$$ R^{\mathcal{D}}+\D^{\mathcal{D}}\mathcal{N}+\frac{1}{2}\,[\mathcal{N}\wedge
\mathcal{N}]=0,$$
and encoding, in particular, $\D^{\mathcal{D}}\mathcal{N}=0$, it follows that
\begin{equation}\label{first}
\D^{\mathcal{D}}\ast\mathcal{N}\Big(\frac{\partial}{\partial z},\frac{\partial}{\partial \bar{z}}\Big)\xi=2i\big(\mathcal{D}_{z}(\mathcal{N}_{\bar{z}}\xi)-\mathcal{N}_{\bar{z}}(\mathcal{D}_{z}\xi)\big)=2i\big(\mathcal{D}_{\bar{z}}(\mathcal{N}_{z}\xi)-\mathcal{N}_{z}(\mathcal{D}_{\bar{z}}\xi)\big).
\end{equation}
On the other hand,
\begin{equation}\label{second}
\begin{split}
2[q_{\infty}\wedge\ast\mathcal{N}]\Big(\frac{\partial}{\partial z},\frac{\partial}{\partial \bar{z}}\Big)\xi
&=2i\big((q_{\infty})_{z}(\mathcal{N}_{\bar{z}}\xi)+(q_{\infty})_{\bar{z}}(\mathcal{N}_{z}\xi)-\mathcal{N}_{z}((q_{\infty})_{\bar{z}}\xi)-\mathcal{N}_{\bar{z}}((q_{\infty})_{z}\xi)\big).
\end{split}
\end{equation}
We will start by proving that both (\ref{first}) and (\ref{second}) vanish, whenever $\xi\in\Gamma(\Lambda^{(1)})$.

As for (\ref{first}), we have
$$\mathcal{D}_{z}(\mathcal{N}_{\bar{z}}\sigma_{\infty})-\mathcal{N}_{\bar{z}}(\mathcal{D}_{z}\sigma_{\infty})=-\mathcal{N}_{\bar{z}}((\sigma_{\infty})_{z})=0$$
and, as $\mathcal{D}_{z}(\sigma_{\infty})_{z}$ takes values in $\Lambda^{1,0}$,
$$\mathcal{D}_{z}(\mathcal{N}_{\bar{z}}(\sigma_{\infty})_{z})-\mathcal{N}_{\bar{z}}(\mathcal{D}_{z}(\sigma_{\infty})_{z})=-\mathcal{N}_{\bar{z}}(\mathcal{D}_{z}(\sigma_{\infty})_{z})=0.$$ Similarly we get
$$\mathcal{D}_{\bar{z}}(\mathcal{N}_{z}(\sigma_{\infty})_{\bar{z}})-\mathcal{N}_{z}(\mathcal{D}_{\bar{z}}(\sigma_{\infty})_{\bar{z}})=0.$$

On the other hand, since $(\mathcal{N}\xi,\sigma_{\infty})=-(\xi,\mathcal{N}\sigma_{\infty})=0$, for every $\xi\in\Gamma(S^{\perp})$, we have $\mathcal{N}\in\Omega^{2}(\Lambda^{(1)}\wedge S^{\perp})$ and also $\ast\mathcal{N}\in\Omega^{2}(\Lambda^{(1)}\wedge S^{\perp})$. Having in consideration that $$[T,a\wedge b]=(Ta)\wedge b+a\wedge(Tb),$$ for all $T\in o(\mathbb{R}^{n+1,1})$ and $a,b\in\mathbb{R}^{n+1,1}$, we conclude that $[q_{\infty}\wedge\ast\mathcal{N}]\in\Omega^{2}(\Lambda\wedge S^{\perp})$ and, therefore, that (\ref{second}) vanishes for all $\xi\in\Lambda^{(1)}$.

The proof is therefore complete if we establish the equality between (\ref{first}) and (\ref{second}) for $\xi=v_{\infty}^{T}$. Since
$$\mathcal{N}v_{\infty}^{T}=-\mathcal{D}v_{\infty}^{\perp}=0$$
and
$$\mathcal{N}_{z}(\mathcal{N}_{\bar{z}}v_{\infty}^{\perp})=\mathcal{N}_{z}((v_{\infty}^{\perp})_{\bar{z}})=\pi_{S^{\perp}}((v_{\infty}^{\perp})_{\bar{z}z})
=\pi_{S^{\perp}}((v_{\infty}^{\perp})_{z\bar{z}})=\mathcal{N}_{\bar{z}}(\mathcal{N}_{z}v_{\infty}^{\perp}),$$
we get
\begin{equation*}
\begin{split}
2[q_{\infty}\wedge\ast\mathcal{N}]\Big(\frac{\partial}{\partial z},\frac{\partial}{\partial \bar{z}}\Big)v_{\infty}^{T}
&=-i\big(\mathcal{N}_{z}((\sigma_{\infty},v_{\infty}^{T})\mathcal{N}_{\bar{z}}v_{\infty}^{\perp})+\mathcal{N}_{\bar{z}}((\sigma_{\infty},v_{\infty}^{T})\mathcal{N}_{z}v_{\infty}^{\perp})\big)\\
&=i\big(\mathcal{N}_{z}(\mathcal{N}_{\bar{z}}v_{\infty}^{\perp})+\mathcal{N}_{\bar{z}}(\mathcal{N}_{z}v_{\infty}^{\perp})\big)\\
&=2i\mathcal{N}_{z}(\mathcal{N}_{\bar{z}}v_{\infty}^{\perp})\\
&=-2i\mathcal{N}_{z}(\mathcal{D}_{\bar{z}}v_{\infty}^{T})\\
&=2i\big(\mathcal{D}_{\bar{z}}(\mathcal{N}_{z}v_{\infty}^{T})-\mathcal{N}_{z}(\mathcal{D}_{\bar{z}}v_{\infty}^{T})\big)\\
&=\D^{\mathcal{D}}\ast\mathcal{N}\Big(\frac{\partial}{\partial z},\frac{\partial}{\partial \bar{z}}\Big)v_{\infty}^{T}.
\end{split}
\end{equation*}
\end{proof}

\begin{proposition}\label{pcqtype1}
Suppose $(\Lambda,q)$ is a constrained Willmore surface. A Laurent polynomial $p(\lambda)$ satisfying the conditions
(\ref{polynomial}) for $d=1$ is a polynomial conserved quantity of
$(\Lambda,q)$ if and only if $$\D p(1)=0,
\;\mathcal{D}^{0,1}p_{1}=0,
\;\mathcal{N}^{1,0}p_{1}+q^{1,0}p_{0}=0.$$
\end{proposition}
\begin{proof}
By Proposition \ref{consequences}, we are left to prove that these
three equations establish $p(\lambda)$ as a polynomial conserved
quantity of $(\Lambda,q)$. For that, suppose that $\D p(1)=0,
\;\mathcal{D}^{0,1}p_{1}=0\mbox{ and
}\mathcal{N}^{1,0}p_{1}+q^{1,0}p_{0}=0$ and first note that,
according to Proposition \ref{iff condition}, it is enough to
establish
$$\mathcal{N}^{1,0}p_{1}+q^{1,0}p_{0}=0,\,\,\mathcal{D}p_{1}+\mathcal{N}^{1,0}p_{0}=0$$ and $$\mathcal{D}p_{0}+\mathcal{N}^{1,0}p_{-1}+\mathcal{N}^{0,1}p_{1}-qp_{0}=0.$$
Considering orthogonal projections onto $S$ and $S^{\perp}$, $\D
v_{\infty}=0$ gives
\begin{equation}\label{A}
\mathcal{D}(p_{-1}+p_{1})+\mathcal{N}p_{0}=0
\end{equation}
and
\begin{equation}\label{B}
\mathcal{N}(p_{-1}+p_{1})+\mathcal{D}p_{0}=0.
\end{equation}
From (\ref{A}), we get
$\mathcal{D}^{1,0}(p_{-1}+p_{1})+\mathcal{N}^{1,0}p_{0}=0$. But
$\mathcal{D}^{0,1}p_{1}=0$ or, equivalently,
$\mathcal{D}^{1,0}p_{-1}=0$. Hence
$\mathcal{D}^{1,0}(p_{-1}+p_{1})=\mathcal{D}p_{1}$ and then
$\mathcal{D}p_{1}+\mathcal{N}^{1,0
}p_{0}=0$. Finally, in view of
(\ref{B}),
$$\mathcal{D}p_{0}+\mathcal{N}^{1,0}p_{-1}+\mathcal{N}^{0,1}p_{1}-qp_{0}=-\mathcal{N}(p_{-1}+p_{1})+\mathcal{N}^{1,0}p_{-1}+\mathcal{N}^{0,1}p_{1}-qp_{0},$$
which implies
\begin{equation*}
\begin{split}
\mathcal{D}p_{0}+\mathcal{N}^{1,0}p_{-1}+\mathcal{N}^{0,1}p_{1}-qp_{0}&=-\mathcal{N}^{0,1}p_{-1}-\mathcal{N}^{1,0}p_{1}-qp_{0}
\\&=-2\Ree(\mathcal{N}^{1,0}p_{1}+q^{1,0}p_{0})\\&=0
\end{split}
\end{equation*}
and completes the
proof.
\end{proof}

\begin{remark} If $(\Lambda,q)$ is a constrained Willmore surface and $p(\lambda)$ is a polynomial conserved quantity of
type $1$ of $(\Lambda,q)$, then
$$q=\sigma_{\infty}\wedge(\mathcal{N}^{1,0}p_{1}+\mathcal{N}^{0,1}p_{-1})=2\sigma_{\infty}\wedge\mbox{Re}(\mathcal{N}^{1,0}p_{1}).$$
As a matter of fact, $q=\sigma_{\infty}\wedge\vartheta$, for some
$\vartheta\in\Omega^{1}(\Lambda^{(1)})$, so that $q
p_{0}=-\vartheta-(\vartheta,p_{0})\sigma_{\infty}$, as
$$(\sigma_{\infty},p_{0})=(\sigma_{\infty},v_{\infty}^{T})=(\sigma_{\infty},v_{\infty})=-1.$$ Hence $$q=-\sigma_{\infty}\wedge(q
p_{0}+(\vartheta,p_{0})\sigma_{\infty})=-\sigma_{\infty}\wedge q
p_{0}.$$ The result is now an immediate consequence of
$$q p_{0}=q^{1,0}p_{0}+q^{0,1}p_{0}=-\mathcal{N}^{1,0}p_{1}-\mathcal{N}^{0,1}p_{-1}=-2\mbox{Re}(\mathcal{N}^{1,0}p_{1}).$$
\end{remark}

\begin{theorem}\label{characterisation of H-parallel surfaces} $\Lambda$ is a constrained Willmore surface admitting a polynomial conserved
quantity $p(\lambda)$ of type $1$ with parallel top term and $p(1)=v_{\infty}$ if and only if the surface $\sigma_{\infty}$ defined by $\Lambda$ in
$S_{v_{\infty}}$ has parallel mean curvature vector.

\begin{proof}
Suppose first that the surface $\Lambda$ admits parallel mean
curvature vector in the space-form $S_{v_{\infty}}$. According to Proposition \ref{multiplier}, we get then
the constrained Willmore surface $(\Lambda,q_{\infty})$, for
$q_{\infty}:=\frac{1}{2}\sigma_{\infty}\wedge\mathcal{N}v_{\infty}^{\perp}$.
Considering
\begin{equation}\label{polynomial_infty}
p(\lambda):=\frac{1}{2}v_{\infty}^{\perp}\lambda^{-1}+v_{\infty}^{T}+\frac{1}{2}v_{\infty}^{\perp}\lambda,
\end{equation}
we have $p(1)=v_{\infty}$ constant and $\mathcal{D}p_{1}=\frac{1}{2}\mathcal{D}v_{\infty}^{\perp}=0$. Furthermore, $$\mathcal{N}^{1,0}p_{1}+(q_{\infty})^{1,0}p_{0}=0,$$ as $\mathcal{N}v_{\infty}^{T}=-\mathcal{D}v_{\infty}^{\perp}=0$, and then
$$\mathcal{N}p_{1}+q_{\infty}p_{0}=\frac{1}{2}(\mathcal{N}v_{\infty}^{\perp}+(\sigma_{\infty},v_{\infty}^{T})\mathcal{N}v_{\infty}^{\perp}-(\mathcal{N}v_{\infty}^{\perp},v_{\infty}^{T})\sigma_{\infty})
=\frac{1}{2}(v_{\infty}^{\perp},\mathcal{N}v_{\infty}^{T})\sigma_{\infty}=0.$$
Hence $p(\lambda)$ is a polynomial conserved quantity of
$(\Lambda,q_{\infty})$ with parallel top term.

Conversely, assuming that a constrained Willmore surface
$(\Lambda,q)$ admits a polynomial conserved quantity
$p(\lambda):=p_{-1}\lambda^{-1}+p_{0}+p_{1}\lambda$ of type $1$, with $p_{1}$ parallel and $p(1)=v_{\infty}$, we obtain also $p_{-1}=\overline{p_{1}}$ parallel, and therefore $$\mathcal{D}v_{\infty}^{\perp}=\mathcal{D}(p_{-1}+p_{1})=0.$$ Consequently, $\sigma_{\infty}$ has parallel mean curvature vector.
\end{proof}
\end{theorem}

\begin{remark}\label{realityoftopterm}
In the proof of Theorem \ref{characterisation of H-parallel surfaces}, we verified, in particular, that, if $\Lambda$ has parallel mean curvature vector in some space-form, then $\Lambda$ admits a polynomial conserved quantity of type $1$ with real top term, given by (\ref{polynomial_infty}). Conversely, given a polynomial conserved quantity $p(\lambda):=p_{-1}\lambda^{-1}+p_{0}+p_{1}\lambda$ of type $1$, with $p_{1}$ real and $p(1)=v_{\infty}$, of a constrained Willmore surface $(\Lambda,q)$, we have $\mathcal{D}^{0,1}p_{1}=0$. Since $p_{1}$ is
real, we get also $\mathcal{D}^{1,0}p_{1}=0$ and then $p_{1}$ parallel. We conclude that $\Lambda$ has parallel mean curvature vector in $S_{v_{\infty}}$ if and only if it is a constrained Willmore surface admitting a polynomial conserved quantity of type $1$ with real top term and $p(1)=v_{\infty}$.
\end{remark}

\begin{corollary} $\Lambda$ is a constrained Willmore surface admitting a polynomial conserved
quantity $p(\lambda)$ of type $1$ with imaginary top term and $p(1)=v_{\infty}$ if and only if the surface $\sigma_{\infty}$ defined by $\Lambda$ in
$S_{v_{\infty}}$ is minimal.
\begin{proof}
First note that $\mathbf{H}=0$ if and only if $v_{\infty}^{\perp}=0$, by \eqref{QHvinftperp}. Considering now a polynomial conserved quantity $p(\lambda)$ of type $1$ with imaginary top term and $p(1)=v_{\infty}$, we get $v_{\infty}^{\perp}=\overline{p_{1}}+p_{1}=0$ and then $\sigma_{\infty}$ minimal. Conversely, if $\sigma_{\infty}$ is minimal then, according to the proof of Theorem \ref{characterisation of H-parallel surfaces}, $(\Lambda, q_{\infty})$ is a constrained Willmore surface, for $q_{\infty}=0$, admitting $p(\lambda)=v_{\infty}^{T}$ as a polynomial conserved quantity of type $1$.
\end{proof}
\end{corollary}

In codimension $1$, Theorem \ref{characterisation of H-parallel surfaces} reads that the constant mean curvature surfaces in
some space-form are exactly the constrained Willmore surfaces which
admit a polynomial conserved quantity of type $1$ with parallel top term.
Actually, the extra condition of parallel top term can be omitted:

\begin{theorem} \label{cmcH}
Let $n=3$. Then $\Lambda$ is a constrained Willmore surface admitting a polynomial conserved
quantity $p(\lambda)$ of type $1$ with $p(1)=v_{\infty}$ if and only if the surface $\sigma_{\infty}$ defined by $\Lambda$ in
$S_{v_{\infty}}$ is a constant mean curvature surface. Furthermore, if $p(\lambda)=p_{-1}\lambda^{-1}+p_{0}+p_{1}\lambda$ is a polynomial conserved quantity of type $1$ of $\Lambda$ with $p(1)=v_{\infty}$, then the constant mean curvature $H$ of $\Lambda$ in $S_{v_{\infty}}$ satisfies $$H^{2}=(v_{\infty}^{\perp},v_{\infty}^{\perp})=4(\Ree(p_{1}),\Ree(p_{1})).$$
\begin{proof}
Suppose that $\Lambda$ is a constrained Willmore
surface admitting a polynomial conserved quantity
$p(\lambda)=p_{-1}\lambda^{-1}+p_{0}+p_{1}\lambda$ of type $1$.
Since $S^{\perp}$ has rank $1$, consider a unit parallel section $N$ of $S^{\perp}$. Take
$\beta\in\Gamma(\mathbb{C})$ such that $p_{1}=\beta N$. Since
$\beta^{2}=(p_{1},p_{1})$ is constant, we conclude that $\beta$ is constant, which implies that $p_1=\beta N$ is parallel.

Finally, taking into account that $v_{\infty}^{\perp}=-HN$, with
$N=H\sigma_{\infty}+\xi$, $\xi\in\Gamma(V_{v_{\infty}}^{\perp})$ unitary and
$H$ the constant mean curvature of $\sigma_{\infty}$ with respect to $\xi$, we get
automatically $H^{2}=(v_{\infty}^{\perp},v_{\infty}^{\perp})$.
 \end{proof}
\end{theorem}

\section{Transformations of special constrained Willmore surfaces}

The class of constrained Willmore surfaces of any given type is preserved by both
spectral deformation and B\"{a}cklund transformation, for special
choices of parameters, as we shall see next. Both constrained Willmore spectral deformation and B\"{a}cklund transformation prove to preserve also the parallelism of the top term of a polynomial conserved quantity. For the particular case of type $1$, this defines  transformations of surfaces with parallel mean curvature into new ones (and, in the particular case of codimension $1$, transformations of surfaces with constant mean curvature vector into new ones).

Let $(\Lambda,q)$ be a special constrained Willmore surface of type $d$.

\begin{theorem}\label{specCWwcq}
Let $\mu$ be in $S^{1}$ and
$\phi^{\mu}_{q}:(\underline{\R}^{n+1,1},\D^{\mu}_{q})\rightarrow
(\underline{\R}^{n+1,1},\D)$ be an isometry of bundles, preserving connections. Suppose that
$p(\lambda)$ is a polynomial conserved quantity of type $d$ of
$(\Lambda,q)$ with $p(\mu)$ non-zero. Then $\phi^{\mu}_{q}p(\mu\lambda)$ is a polynomial
conserved quantity of type $d$ of the spectral deformation
$(\phi^{\mu}_{q}\Lambda, \mathrm{Ad}_{\phi^{\mu}_{q}}(q_{\mu}))$, of
parameter $\mu$, of $\Lambda$.
\end{theorem}
\begin{proof}
Write $p(\lambda)=\sum_{k=-d}^d p_{k}\lambda^{k}$.
By hypothesis,
$$v_{\infty}^{\mu}:=\phi^{\mu}_{q}(\sum_{k=-d}^d p_{k}\mu^{k})=\phi^{\mu}_{q}p(\mu)$$
is non-zero. On the other hand, as $\phi^{\mu}_{q}$ is real and
$\mu$ is unit, we have
$$\mu^{-k}\phi^{\mu}_{q}p_{-k}=\overline{\mu^{k}\,
\phi^{\mu}_{q}p_{k}}.$$ Having in consideration that $\phi^{\mu}_{q}$
is an isometry, and, in particular,
$(\phi^{\mu}_{q}S)^{\perp}=\phi^{\mu}_{q}S^{\perp}$, and that $$S_{\phi^{\mu}_{q}\Lambda}=\phi^{\mu}_{q}S_{\Lambda},$$ we conclude
that
$$\phi^{\mu}_{q}p(\mu\lambda)=\sum_{k=-d}^d(\mu^{k}\phi^{\mu}_{q}p_{k})\lambda^{k}$$ is of the right form.
The fact that
$\phi^{\mu}_{q}:(\underline{\R}^{n+1,1},\D^{\mu}_{q})\rightarrow
(\underline{\R}^{n+1,1},\D)$ preserves connections, and,
consequently,
$$\D^{\lambda,\mathrm{Ad}_{\phi^{\mu}_{q}}(q_{\mu})}_{\phi^{\mu}_{q}\Lambda}=\phi^{\mu}_{q}\circ
(\D^{\mu,q}_{\Lambda})^{\lambda,q_{\mu}}_{\Lambda}\circ
(\phi^{\mu}_{q})^{-1}=\phi^{\mu}_{q}\circ
\D^{\mu\lambda,q}_{\Lambda}\circ (\phi^{\mu}_{q})^{-1},$$ for all $\lambda\in \C\backslash\{0\}$, completes
the proof.
\end{proof}

B\"{a}cklund transformations of constrained Willmore surfaces
preserve the existence of a polynomial conserved quantity of the
same type, in the following terms:
\begin{theorem}\label{BTsofCMCs}
Suppose $p(\lambda)$ is a polynomial conserved quantity of type $d$
of $(\Lambda,q)$. Suppose $\alpha,L$ are B\"{a}cklund
transformation parameters to $(\Lambda,q)$ with
\begin{equation}\label{eq:palhopahaortogLalpha}
p(\alpha)\perp \overline{L}.
\end{equation}
Then
$$\hat{p}(\lambda):=r(1)^{-1}r(\overline{\lambda}\,^{-1})p(\lambda)$$
is a polynomial conserved quantity of type $d$ of the B\"{a}cklund
transform $(\hat{\Lambda},\hat{q})$ of $(\Lambda,q)$ of parameters
$\alpha,L$.
\end{theorem}

To prove the theorem, we start by establishing an alternative expression for the dressing gauge $r$:
\begin{lemma}\label{rstarvshatrstar}
Suppose $\alpha,L$ are B\"{a}cklund
transformation parameters to $(\Lambda,q)$. Then
\begin{equation}\label{eq:rel}
r=Kp_{\hat{\alpha},\overline{\tilde{L}}}p_{\alpha,L}^{-},
\end{equation}
for $$K:=p_{\hat{\alpha},\overline{L}}(0)p_{\hat{\alpha},\overline{\tilde{L}}}(0).$$
\end{lemma}

The proof of the lemma we present next will be based on the
following:

\begin{lemma}\cite{Bur06}\label{TrioDeHolomorfia}
Let
$$\gamma(\lambda)=\lambda\,\pi_{L_{1}}+\pi_{L_{0}}+\lambda^{-1}\,\pi_{L_{-1}}$$
and
$$\hat{\gamma}(\lambda)=\lambda\,\pi_{\hat{L}_{1}}+\pi_{\hat{L}_{0}}+\lambda^{-1}\,\pi_{\hat{L}_{-1}}$$ be homomorphisms of $\C^{n+2}$ corresponding to decompositions $$\C^{n+2}=L_{1}\oplus L_{0}\oplus L_{-1}=\hat{L}_{1}\oplus \hat{L}_{0}\oplus
 \hat{L}_{-1}$$ with $L_{\pm 1}$ and $\hat{L}_{\pm 1}$ null lines and
 $L_{0}=(L_{1}\oplus L_{-1})^{\perp}$, $\hat{L}_{0}=(\hat{L}_{1}\oplus
 \hat{L}_{-1})^{\perp}$. Suppose $\mathrm{Ad}\,\gamma$ and $\mathrm{Ad}\,\hat{\gamma}$ have simple poles. Suppose as well that $\xi$ is a map into
 $O(\C^{n+2})$ holomorphic near $0$ such that $L_{1}=\xi
 (0)\hat{L}_{1}$. Then $\gamma\xi\hat{\gamma}^{-1}$ is holomorphic and invertible at $0$.
\end{lemma}

Next we prove Lemma \ref{rstarvshatrstar}:

\begin{proof}
In view of
$L=p_{\hat{\alpha},\overline{L}}(\alpha)^{-1}\tilde{L}$,
after an appropriate change of variable, we conclude, by Lemma \ref
{TrioDeHolomorfia}, that $p^{-}_{\alpha,L}\,
p_{\hat{\alpha},\overline{L}}^{-1}\,
(p^{-}_{\alpha,\tilde{L}})^{-1}$ admits a holomorphic and
invertible extension to $\mathbb{P}^{1}\backslash\{\pm
\hat{\alpha},-\alpha\}$. On the other hand, in view of \eqref{eq:rhopq}, the holomorphicity and invertibility of
$p^{-}_{\alpha,L}\, p_{\hat{\alpha},\overline{L}}^{-1}\,
(p^{-}_{\alpha,\tilde{L}})^{-1}$ at the points $\alpha$
and $-\alpha$ are equivalent. Thus $p^{-}_{\alpha,L}\,
p_{\hat{\alpha},\overline{L}}^{-1}\,
(p^{-}_{\alpha,\tilde{L}})^{-1}$ admits a holomorphic and
invertible extension to $\mathbb{P}^{1}\backslash\{\pm \hat{\alpha}\}$, and
so does, therefore, $(p^{-}_{\alpha,L}\,
p_{\hat{\alpha},\overline{L}}^{-1}\,
(p^{-}_{\alpha,\tilde{L}})^{-1})^{-1}\,
p_{\hat{\alpha},\overline{L}}^{-1}$. A similar argument shows that
$p^{-}_{\alpha,\tilde{L}}\, (p_{\hat{\alpha},\overline{L}}\,
(p^{-}_{\alpha,L})^{-1}\,
p_{\hat{\alpha},\overline{\tilde{L}}}^{-1})$ admits a holomorphic
extension to $\mathbb{P}^{1}\backslash\{\pm\alpha\}$. But
$$p^{-}_{\alpha,\tilde{L}}\,p_{\hat{\alpha},\overline{L}}\,
(p^{-}_{\alpha,L})^{-1}\,
p_{\hat{\alpha},\overline{\tilde{L}}}^{-1}=(p^{-}_{\alpha,L}\,
p_{\hat{\alpha},\overline{L}}^{-1}\,
(p^{-}_{\alpha,\tilde{L}})^{-1})^{-1}\,
p_{\hat{\alpha},\overline{L}}^{-1}.$$ We conclude that
$p^{-}_{\alpha,\tilde{L}}\,p_{\hat{\alpha},\overline{L}}\,
(p^{-}_{\alpha,L})^{-1}\,
p_{\hat{\alpha},\overline{\tilde{L}}}^{-1}$ extends holomorphically to
$\mathbb{P}^{1}$ and is, therefore, constant. Evaluating at
$\lambda=0$ gives
$$p^{-}_{\alpha,\tilde{L}}\,p_{\hat{\alpha},\overline{L}}\,
(p^{-}_{\alpha,L})^{-1}\,
p_{\hat{\alpha},\overline{\tilde{L}}}^{-1}=p_{\hat{\alpha},\overline{L}}(0)\,p_{\hat{\alpha},
\overline{\tilde{L}}}(0),$$ completing the proof.
\end{proof}

We proceed now to the proof of Theorem \ref{BTsofCMCs}:

\begin{proof}
Consider projections
$\pi_{\overline{L}}:\underline{\C}^{n+2}\rightarrow
\overline{L}$, $\pi_{(\overline{L}\oplus\rho \overline{L})^{\perp}}:\underline{\C}^{n+2}\rightarrow
(\overline{L}\oplus\rho \overline{L})^{\perp}$ and $\pi_{\rho
\overline{L}}:\underline{\C}^{n+2}\rightarrow \rho
\overline{L}$ with respect to the decomposition
$$\underline{\C}^{n+2}=\overline{L}\oplus(\overline{L}\oplus\rho
\overline{L})^{\perp}\oplus\rho
\overline{L}.$$ Since
$\overline{L}$ and $\rho \overline{L}$ are never orthogonal, condition
\eqref{eq:palhopahaortogLalpha} establishes, in particular,
$\pi_{\rho \overline{L}}p(\alpha)=0$. On the other hand, in view of \eqref{eq:2} and \eqref{eq:3}, we have
\begin{equation}\label{eq:rhopinf}
\rho p(\lambda)=(-1)^{d+1}p(-\lambda)
\end{equation}
for all $\lambda$. Hence
$$\pi_{\overline{L}}p(-\alpha)=(-1)^{d+1}\pi_{\overline{L}}\rho p(\alpha)=(-1)^{d+1}\rho\pi_{\rho
\overline{L}}p(\alpha)=0.$$ It follows that
$$p_{\hat{\alpha},\overline{L}}(\overline{\lambda}\,^{-1})\,p(\lambda)=\frac{\overline{\lambda}\,^{-1}-\overline{\alpha}\,^{-1}}{\overline{\lambda}\,^{-1}+\overline{\alpha}\,^{-1}}\,\pi_{\overline{L}}p(\lambda)+\pi_{(\overline{L}\oplus\rho \overline{L})^{\perp}}p(\lambda)+\frac{\overline{\lambda}\,^{-1}+\overline{\alpha}\,^{-1}}{\overline{\lambda}\,^{-1}-\overline{\alpha}\,^{-1}}\,\pi_{\rho
\overline{L}}p(\lambda)$$ has no poles and, therefore, that $$\hat{p}(\lambda)=r(1)^{-1}p^{-}_{\alpha,\tilde{L}}(\overline{\lambda}\,^{-1})p_{\hat{\alpha},\overline{L}}(\overline{\lambda}\,^{-1})\,p(\lambda)$$
has, at most, poles at $\lambda=\pm\alpha$.

Consider now projections
$\pi_{L}:\underline{\C}^{n+2}\rightarrow L$, $\pi_{(L\oplus\rho L)^{\perp}}:\underline{\C}^{n+2}\rightarrow
(L\oplus\rho L)^{\perp}$ and
$\pi_{\rho L}:\underline{\C}^{n+2}\rightarrow\rho
L$ with respect to the
decomposition $$\underline{\C}^{n+2}=L\oplus(L\oplus\rho L)^{\perp}\oplus\rho
L.$$ By \eqref{eq:1}, we have
\begin{equation}\label{eq:pinfconjugation}
\overline{p(\lambda)}=p(\overline{\lambda}\,^{-1}),
\end{equation}
for all $\lambda$. In particular, $\overline{p(\alpha)}=p(\hat{\alpha})$ and, therefore, condition \eqref{eq:palhopahaortogLalpha} establishes $p(\hat{\alpha})\in\Gamma(L^{\perp})$. Since
$L$ and $\rho L$ are never orthogonal, we conclude that
$\pi_{\rho L}p(\hat{\alpha})=0$. Hence, by \eqref{eq:rhopinf}, $$\pi_{L}p(-\hat{\alpha})=(-1)^{d+1}\pi_{L}\rho p(\hat{\alpha})=(-1)^{d+1}\rho\pi_{\rho
L}p(\hat{\alpha})=0.$$
It follows that
$$p^{-}_{\alpha,L}(\overline{\lambda}\,^{-1})\,p(\lambda)=\frac{\alpha-\overline{\lambda}\,^{-1}}{\alpha+\overline{\lambda}\,^{-1}}\,\pi_{L}p(\lambda)+\pi_{(L\oplus\rho L)^{\perp}}p(\lambda)+\frac{\alpha+\overline{\lambda}\,^{-1}}{\alpha-\overline{\lambda}\,^{-1}}\,\pi_{\rho
L}p(\lambda)$$ has no poles and, then, by Lemma \ref{rstarvshatrstar}, that $$\hat{p}(\lambda)=r(1)^{-1}Kp_{\hat{\alpha},\overline{\tilde{L}}}(\overline{\lambda}\,^{-1})p^{-}_{\alpha,L}(\overline{\lambda}\,^{-1})\,p(\lambda)$$
has, at most, poles at $\lambda=\pm\hat{\alpha}$. We conclude that
$\hat{p}(\lambda)$ has no poles. Write $p(\lambda)=\sum_{k=-d}^d p_{k}\lambda^{k}$. The fact that
$$\mathrm{lim}_{\lambda\rightarrow
\infty}\,\lambda^{-d}\hat{p}(\lambda)=r(1)^{-1}\,r(0)\,p_{d}$$
and $$\mathrm{lim}_{\lambda\rightarrow 0}\,\lambda^{d}
\hat{p}(\lambda)=r(1)^{-1}\,r(\infty)\,\overline{p_{d}}$$ are both
finite establishes then $\hat{p}(\lambda)$ as a Laurent polynomial
with degree smaller or equal to $d$.

Now let $\hat{\rho}$ denote reflection across $\hat{S}$. According to \eqref{eq:rhopq}, we have
\begin{equation}\label{eq:rrho}
r(-\lambda)=\rho r(\lambda)\rho^{-1},
\end{equation}
for all $\lambda\in\mathbb{P}^{1}$, so that
$$\hat{\rho}\hat{p}(\lambda)=r(1)^{-1}\rho\, r(1)\hat{p}(\lambda)=r(1)^{-1}\rho
\,r(\overline{\lambda}\,^{-1})\,p(\lambda)=r(1)^{-1}r(-\overline{\lambda}\,^{-1})\rho
p(\lambda)$$ and, therefore, following \eqref{eq:rhopinf},
$$\hat{\rho}\hat{p}(\lambda)=(-1)^{d+1}\hat{p}(-\lambda),$$
showing that the coefficients on $\lambda^{k}$ in $\hat{p}(\lambda)$
are sections of $\hat{S}^{\perp}$ if $k$ has the same parity as $d$, being, otherwise,  sections of $\hat{S}$.

Next we verify that
$\overline{\hat{p}(\lambda)}=\hat{p}(\overline{\lambda}\,^{-1})$,
equivalent to the complex conjugation conditions on the coefficients
in $\hat{p}(\lambda)$. For that, observe that, by Lemma \ref{rstarvshatrstar},
\begin{equation*}
\begin{split}
\overline{r(\lambda)}&=\overline{p^{-}_{\alpha,\tilde{L}}(\lambda)p_{\hat{\alpha},\overline{L}}(\lambda)}\\
&=p^{-}_{\overline{\alpha},\overline{\tilde{L}}}(\overline{\lambda})p_{\alpha^{-1},L}(\overline{\lambda})\\
&=p_{\hat{\alpha},\overline{\tilde{L}}}(\overline{\lambda}\,^{-1})p^{-}_{\alpha,L}(\overline{\lambda}\,^{-1})\\
&=K^{-1}r(\overline{\lambda}\,^{-1}),
\end{split}
\end{equation*}
as well as, on the other hand,
\begin{equation*}
\begin{split}
\overline{r(\lambda)}&=\overline{Kp_{\hat{\alpha},\overline{\tilde{L}}}(\lambda)p^{-}_{\alpha,L}(\lambda)}\\&=\overline{K}p_{\alpha^{-1},\tilde{L}}(\overline{\lambda})p^{-}_{\overline{\alpha},\overline{L}}(\overline{\lambda})\\&=\overline{K}p^{-}_{\alpha,\tilde{L}}(\overline{\lambda}\,^{-1})p_{\hat{\alpha},\overline{L}}(\overline{\lambda}\,^{-1})\\&=\overline{K}r(\overline{\lambda}\,^{-1}),
\end{split}
\end{equation*}
for all $\lambda\in\C\backslash\{0,\pm\alpha\}$. In particular, $\overline{r(1)^{-1}}=r(1)^{-1}K$. The conclusion now follows immediately from \eqref{eq:pinfconjugation}.

Finally, note that
$$\D^{\lambda,\hat{q}}_{\hat{S}}\hat{p}(\lambda)=r(1)^{-1}\circ r(\lambda)\circ{\D}^{\lambda,q}_{S}\circ r(\lambda)^{-1}\circ
r(\overline{\lambda}\,^{-1})\,p(\lambda)=r(1)^{-1}r(\lambda)\circ
\D^{\lambda,q}_{S}\,p(\lambda)=0,$$ for
$\lambda\in S^{1}$, which completes the proof (since $\D^{\lambda,\hat{q}}_{\hat{S}}\hat{p}(\lambda)$ is a polynomial with an infinite number of zeros).
\end{proof}

Following Theorem \ref{specCWwcq} and Theorem \ref{BTsofCMCs}, we have, furthermore:

\begin{theorem}
Both constrained Willmore spectral deformation and B\"{a}cklund transformation preserve the parallelism of the top term of a polynomial conserved quantity, for special choices of parameters.
\end{theorem}

\begin{proof}
Suppose first that we are in the conditions of Theorem \ref{specCWwcq}. Write  $p(\lambda)=\sum_{k=-d}^d p_{k}\lambda^{k}$. Let $\pi_{S^{\perp}}$ and $\pi_{\phi^{\mu}_{q}S^{\perp}}$ denote the orthogonal projections of $\underline{\R}^{n+1,1}$ onto $S^{\perp}$ and $(S_{\phi^{\mu}_{q}\Lambda})^{\perp}=\phi^{\mu}_{q}S^{\perp}$, respectively. Suppose that $p_{d}\in\Gamma(S^{\perp})$ is parallel, $\pi_{S^{\perp}}\circ \D p_{d}=0$, and let us prove that then so is $\mu^{d}\phi^{\mu}_{q}p_{d}\in\Gamma((S_{\phi^{\mu}_{q}\Lambda})^{\perp})$, $$\pi_{\phi^{\mu}_{q}S^{\perp}}\circ \D(\mu^{d}\phi^{\mu}_{q}p_{d})=0.$$ For that, note that, as $p_{d}\in\Gamma(S^{\perp})$, we have $\mathcal{N}p_{d}\in\Gamma(S)$ and, on the other hand, $qp_{d}=0$, since $q\in\Omega^{1}(\Lambda\wedge\Lambda^{(1)})$. Thus $$
\pi_{S^{\perp}}\circ \D^{\mu}_{q}p_{d}=\mathcal{D}p_{d}=\pi_{S^{\perp}}\circ \D p_{d}=0.$$
The fact that
$\phi^{\mu}_{q}:(\underline{\R}^{n+1,1},\D^{\mu}_{q})\rightarrow
(\underline{\R}^{n+1,1},\D)$ preserves connections establishes then
$$\pi_{\phi^{\mu}_{q}S^{\perp}}\circ \D\circ \phi^{\mu}_{q}p_{d}=\phi^{\mu}_{q}\circ \pi_{S^{\perp}}\circ \D^{\mu}_{q}p_{d}=0.$$
The conclusion follows, by the constancy of $\mu$.

Suppose now that we are in the conditions of Theorem \ref{BTsofCMCs} for  $p(\lambda)=\sum_{k=-d}^d p_{k}\lambda^{k}$. Suppose, again, that $p_{d}\in\Gamma(S^{\perp})$ is parallel, $\mathcal{D}p_{d}=0$, and let us prove that then so is $$\hat{p}_{d}:=\mathrm{lim}_{\lambda\rightarrow
\infty}\,\lambda^{-d}\hat{p}(\lambda)=r(1)^{-1}\,r(0)\,p_{d},$$
the top term of $\hat{p}(\lambda)$, $\mathcal{D}_{\hat{S}}\hat{p}_{d}=0$. In view of Proposition \ref{consequences}, we are left to verify that $\mathcal{D}_{\hat{S}}^{1,0}\hat{p}_{d}=0$ or, equivalently, that $r(1)^{-1}\circ(\mathcal{D}^{\hat{\D}}_{S})^{1,0}\circ r(0)p_{d}=0$, for $$\hat{\D}:=r(1)\circ \D\circ r(1)^{-1}.$$ But $$(\mathcal{D}^{\hat{\D}}_{S})^{1,0}=r(0)\circ(\mathcal{D}^{1,0}-q^{1,0})\circ r(0)^{-1}-\tilde{q}^{1,0}$$ (see \cite{BQ}, Lemma 3.9). Now note that, evaluating \eqref{eq:rrho} at $\lambda=0$ and at $\lambda=\infty$ shows that both $r(0)$ and $r(\infty)$ commute with $\rho$, establishing, in particular, that $$r(0)_{|S^{\perp}},r(\infty)_{|S^{\perp}}\in\Gamma(O(S^{\perp})).$$ The fact that $q\in\Omega^{1}(\Lambda\wedge\Lambda^{(1)})$ vanishes in $\Gamma(S^{\perp})$ together with the parallelism of $p_{d}$ combine to complete the proof.
\end{proof}

For the particular case of $d=1$, it follows that:

\begin{corollary}
The class of parallel mean curvature vector surfaces in space-forms is preserved
under both constrained Willmore spectral deformation and B\"{a}cklund transformation, for special choices of parameters, with preservation of the space-form in the latter case.

The class of constant mean curvature surfaces in $3$-dimensional space-forms is preserved under both constrained Willmore spectral deformation and B\"{a}cklund transformation, for special choices of parameters, with preservation of both the space-form and the mean curvature, in the latter case.
\end{corollary}
\begin{proof}
The preservation of the space-form under B\"{a}cklund transformation is a consequence of the fact that, in the conditions of Theorem \ref{BTsofCMCs}, $\hat{p}(1)=p(1)$. Assuming now that $n=3$, we can take unit sections $N$ and $\hat{N}$ such that $S^{\perp}=\langle N\rangle$ and $\hat{S}^{\perp}=\langle \hat{N}\rangle$. Considering the sections $\beta\in\Gamma(\mathbb{C})$ and $\hat{\beta}\in\Gamma(\mathbb{C})$ such that $p_{1}=\beta N$ and $\hat{p}_{1}=\hat{\beta}\hat{N}$, we get, by virtue of $(\hat{p}(\lambda),\hat{p}(\lambda))=(p(\lambda),p(\lambda))$ (recall Theorem \ref{BTsofCMCs}), that $\hat{\beta}^{2}=(\hat{p}_{1},\hat{p}_{1})=(p_{1},p_{1})=\beta^{2}$ (which is constant). Therefore $\hat{\beta}=\beta$ or $\hat{\beta}=-\beta$. According to Theorem \ref{cmcH}, we obtain the preservation of the mean curvatures of $\Lambda$ and $\hat{\Lambda}$ in $S_{p(1)}=S_{\hat{p}(1)}$, since $$(\Ree(\hat{p}_{1}),\Ree(\hat{p}_{1}))=(\Ree(\hat{\beta}))^{2}=(\Ree(\beta))^{2}=(\Ree(p_{1}),\Ree(p_{1})).$$
\end{proof}

\providecommand{\bysame}{\leavevmode\hbox
to3em{\hrulefill}\thinspace}
\providecommand{\MR}{\relax\ifhmode\unskip\space\fi MR }
\providecommand{\MRhref}[2]{%
  \href{http://www.ams.org/mathscinet-getitem?mr=#1}{#2}
} \providecommand{\href}[2]{#2}

\end{document}